\def\marker{\>\hbox{${\vcenter{\vbox{
					\hrule height 0.4pt\hbox{\vrule width 0.4pt height 6pt
						\kern6pt\vrule width 0.4pt}\hrule height 0.4pt}}}$}\>}

\documentclass[12pt]{article}
\usepackage{indentfirst}
\usepackage{epsfig}
\usepackage{fullpage}
\usepackage{amsmath,amsfonts,amssymb,amsthm,graphics,amscd}
\usepackage{pgf,tikz}
\usepackage{graphicx}
\usepackage{subfigure}
\usepackage{amsmath,amsfonts,amssymb,amsthm,graphics,amscd,pst-plot,pstricks}
\usepackage{graphicx}

\newtheorem {Theorem}  {Theorem}
\newtheorem {Lemma}{Lemma}[section]

\theoremstyle{definition}
\newtheorem{Definition}{Definition}

\newtheorem{Conjecture}{Conjecture}

\newcommand{\D}{\Delta}

\newcommand{\phiv}{\varphi}

\newcommand{\CC}{\mathcal{C}}

\usepackage[
pdfauthor={},
pdftitle={On the Hilton-Zhao vertex-splitting conjecture},
pdfstartview=XYZ,
bookmarks=true,
colorlinks=true,
linkcolor=blue,
urlcolor=blue,
citecolor=blue,
bookmarks=false,
linktocpage=true,
hyperindex=true
]{hyperref}

\begin{document}
\begin{center}

{\Large \bf On the Hilton-Zhao vertex-splitting conjecture}

\vspace{10mm}

Xuli Qi$^*$,\,\,\,\,Yanrui Feng

\vspace{4mm}
\baselineskip=0.25in
{\it Department of Mathematics, Hebei University of Science and Technology,}  \\
{\it Shijiazhuang 050018, P.R. China}\\
\vspace{6mm}
\end{center}

\footnotetext{Corresponding author. E-mail addresses: xuliqi@hebust.edu.cn (X. Qi); 2023110010@stu.hebust.edu.cn (Y. Feng)}
\date{}

\begin{abstract}
Let $G$ be a simple graph with order $n$, maximum degree $\D(G)$, and chromatic index
$\chi'(G)$, respectively.
A graph	$G$ is {\em edge-chromatic critical} if $\chi'(H)<\chi'(G)$ for every proper
subgraph $H$ of $G$. Assume that $G$ is an $n$-vertex connected regular Class $1$ graph,
and let $G^*$ be obtained from $G$ by splitting one vertex into two vertices. Hilton and
Zhao in 1997 proposed the vertex-splitting conjecture: if $\D(G) \textgreater\frac{n}{3}$,
then $G^*$ is edge-chromatic critical. Recently, Cao, Chen, and Shan (Discrete Math. 2022) verified
the conjecture for $\D(G)\ge\frac{3n}{4}$. In this paper, we confirm the
conjecture for $\D(G) \ge\frac{2n-2}{3}$.
		
\end{abstract}

\section{Introduction}
We consider simple connected graphs and generally follow the book~\cite{SSTF} of Stiebitz et al. for notation and terminology.
Let $G$ be a graph, and $V(G)$ and $E(G)$ be the vertex set and the edge set of $G$, respectively.
A graph $G$ is {\bf $\D$-regular} if $\delta(G)=\Delta(G)=\D$, where $\delta(G)$ and $\Delta(G)$ are respectively the minimum degree and maximum degree of $G$.
Let $[k]:=\{1, \dots, k\}$, that is, the set of the first $k$ consecutive positive integers.	
A {\bf $k$-edge-coloring} of a graph $G$ is a map $\phiv:E(G)\to [k]$ that assigns to
every edge $e$ of $G$ a color $\phiv(e)\in [k]$ such that no two adjacent edges of $G$
receive the same color.
	
Denote by $\CC^k(G)$ the set of all $k$-edge-colorings of $G$.
The {\bf  chromatic index} $\chi'(G)$ is the least integer $k$ such that $\CC^k(G)\ne\emptyset$.
In 1960s, Vizing~\cite{Vizing1965} and, independently, Gupta~\cite{Gupta} proved that
$\Delta(G)\le \chi'(G)\le\Delta(G)+1$, which leads to a natural classification of graphs.
Following Fiorini and Wilson~\cite{FW}, if $\chi'(G)=\Delta(G)$, then $G$ is of {\bf Class
1}; otherwise, it is of {\bf Class 2}. Holyer~\cite{Holyer} proved that it is NP-complete to
determine whether an arbitrary graph is of Class $1$. We know that $\chi'(G)\ge
\left|E(G)\right|/\lfloor  \left|V(G)\right|/2\rfloor$ since every matching of $G$ has at
most $\lfloor  \left|V(G)\right|/2\rfloor$ edges. A graph $G$ is {\bf overfull} if
$\left|E(G)\right|> \Delta(G)\lfloor \left|V(G)\right|/2\rfloor$.
Clearly, if $G$ is overfull then $\chi'(G)=\Delta(G)+1$, and so $G$ is of Class $2$.
Easily implied by its definition, overfull graphs have odd order.
	
A graph	$G$ is {\bf edge-chromatic critical} if $\chi'(H)<\chi'(G)$ for every proper
subgraph $H$ of $G$. In investigating the classification problem,  edge-chromatic critical
graphs of Class $2$ are of particular interest.	
Furthermore, we call $G$ a {\bf $\D$-critical} graph if $G$ is edge-chromatic critical
and $\chi'(G)=\D+1$, where $\D:=\D(G)$.
We study sufficient conditions for a Class $2$ graph to be edge-chromatic critical.
For $v \in V(G)$, let $N_G(v)$ denote the set of neighbors of $v$ in $G$. 	
A {\bf vertex-splitting} in a graph $G$ at a vertex $v$ forms a new graph  $G^*$:
replacing $v\in V(G)$ with two new adjacent vertices $v_1$ and $v_2$ and partitioning
$N_G(v)$  into two nonempty subsets that serve as the set of neighbors of $v_1$ and $v_2$
from $V(G)$ in $G^*$, respectively.
We say the new graph $G^*$ is obtained from $G$ by a vertex-splitting.
Vertex-splitting was called the ``M\"obius-type gluing technique" in~\cite{BV} and \cite{Vietri2015}.
It is easy to see that a regular Class $1$ graph has even order, and that every graph
obtained from a regular graph of even order by a vertex-splitting is overfull.
	
Hilton and Zhao~\cite{HZ} in $1997$ proposed the following conjecture.

\begin{Conjecture}[Vertex-splitting conjecture]
Let $G$ be an $n$-vertex connected $\D$-regular Class $1$ graph with $\D
\textgreater\frac{n}{3}$. If $G^*$ is obtained from $G$ by a vertex-splitting, then $G^*$ is
$\D$-critical.
\end{Conjecture}

It is easy to check that the graph $G^*$ above is overfull, so it is of Class $2$.
Thus the difficulty of vertex-splitting conjecture is to test that $G^*$ is  edge-chromatic critical.
Hilton and Zhao~\cite{HZ} verified the conjecture for graphs $G$ with $\D(G) \ge
(\sqrt{7}-1)n/2 \approx0.82n$. Song~\cite{Song} in 2002 verified the conjecture
for a class of special graphs $G$ with $\D(G) \ge 0.5n$. Recently, Cao, Chen and
Shan~\cite{CCS2022} proved the conjecture when $\D(G) \ge 0.75n$. In
this paper, we confirm the conjecture for $\D(G) \ge \frac{2n-2}{3}$ as follows.

\begin{Theorem}\label{T1}
Let $n$ and $\D$ be positive integers such that $\D \ge\frac{2n-4}{3}$. If an $n$-vertex graph $G$ is obtained
from an $(n-1)$-vertex connected $\D$-regular  Class $1$ graph by a vertex-splitting, then
$G$ is $\D$-critical.
\end{Theorem}

\section{Definitions and preliminary results}

Let $G$ be a graph. For $e \in E(G)$, $G-e$ represents the graph obtained from $G$ by
deleting the edge $e$. An edge $e$ is a {\bf critical edge} of $G$ if
$\chi'(G-e)<\chi'(G)$. Clearly, if $G$ is of Class $2$ and every edge is critical, then $G$
is edge-chromatic critical. For $u \in V(G)$, let  $d_G(u)$ denote the degree of $u$  in $G$.
The symbol $\D$ is reserved for $\D(G)$ in this paper. A $\D$-vertex in $G$ is a vertex with maximum degree in $G$.
For $u \in V(G)$, a $\D$-neighbor of  $u$ is a neighbor of $u$ that is a $\D$-vertex in $G$.
When $u,v\in V(G)$ are adjacent, we call $(u,v)$ a {\bf full-deficiency pair} of $G$ if
$d_G(u)+d_G(v)=\Delta+2$.

Let $G$ be a graph with an edge $e\in E(G)$, and an edge coloring $\phiv\in \CC^{k}(G-e)$
for some integer $k\ge\D$. For a vertex $v\in V(G)$, let $\phiv(v)$ denote the set of colors
assigned to edges incident with $v$, and $\overline{\varphi}(v) = [k]\setminus\phiv(v)$,
that is, the set of colors not assigned to any edge incident with $v$. We call $\phiv(v)$ the
set of colors {\bf present} at $v$ and $\overline{\varphi}(v)$ the set of colors {\bf
missing} at $v$. Clearly, $|\phiv(v)| + |\overline{\varphi}(v)| =k$ for each vertex $v\in V(G)$.
For a vertex set $X\subseteq V(G)$, we define $\overline{\varphi}(X)=\bigcup_{v\in X}
\overline{\varphi}(v)$.
The set $X$ is called elementary with respect to $\phiv$ or simply {\bf $\phiv$-elementary}
if $\overline{\varphi}(u)\cap\overline{\varphi}(v)=\emptyset$ for any two distinct vertices
$u,v\in X$.
For a color $\alpha\in [k]$, let  $E_{\varphi,\alpha}(G)$ denote the set of edges colored
with $\alpha$.
Let $\alpha,\beta\in[k]$ be  two distinct colors, and $H$ be the spanning subgraph induced
by $E_{\varphi,\alpha}(G)$ and $E_{\varphi,\beta}(G)$.
Clearly, every component of $H$ is either a path or an even cycle, each of which is referred to as an
{\bf $(\alpha,\beta)$-chain} of $G$.
If we swap the colors $\alpha$ and $\beta$ on an $(\alpha,\beta)$-chain $C$, then we obtain a
new (proper) $k$-edge-coloring of $G$, which is also in $\mathcal{C}^k(G-e)$. This operation is called a {\bf Kempe change}.
Furthermore, if $C$ is a path connecting vertices $u$ and $v$, then we say that an
$(\alpha,\beta)$-chain $C$ has endvertices $u$ and $v$ or say that $u$ and $v$ are {\bf
$(\alpha,\beta)$-linked}.
For a vertex $v$ of $G$, we denote by $P_v(\alpha,\beta,\varphi)$ or simply $P_v(\alpha,\beta)$ the unique
$(\alpha,\beta)$-chain containing  the vertex $v$.
{\em For any two distinct vertices $u$, $v\in V(G)$, the two
chains $P_u(\alpha,\beta,\varphi)$ and $P_v(\alpha,\beta,\varphi)$ are either identical or
disjoint, which is an important fact in our proofs.}
More generally, for an $(\alpha,\beta)$-chain, if it is a path and  contains two vertices
$a$ and $b$, then we let $P_{[a,b]}(\alpha, \beta,\varphi)$ be its subchain  with endvertices $a$ and $b$.
The operation of swapping the colors $\alpha$ and $\beta$ on the subchain $P_{[a,b]}(\alpha,
\beta,\varphi)$ is still called a Kempe change, but the resulting coloring may no longer be a proper
edge coloring.
In particular, we may change the color of one edge in our proofs and adopt the notation $uv:\alpha\rightarrow\beta$ to mean recolor the edge $uv$ from $\alpha$ to $\beta$.
If $x,y \in P_u(\alpha,\beta,\phiv)$ with $x$ positioned between $u$ and $y$ on the path, then we say $P_u(\alpha,\beta,\phiv)$ meets $x$ before $y$.

\begin{Definition}[Multi-fan]
Let $G$ be a graph with an edge $e_1=xy_1$, and an edge coloring $\phiv\in \CC^{k}(G-e_1)$
for some integer $k\ge\D$.  A multi-fan at $x$ with respect to $e_1$ and $\phiv$ is a
sequence $F=(x,e_1,y_1,\dots,e_p,y_p)$ with $p \ge 1$ consisting of edges $e_1, \dots,e_p$
and vertices $x,y_1, \dots ,y_p$ satisfying the following two conditions$\colon$
\begin{itemize}
			\item [{\bf F1.}]	The edges $e_1, \dots, e_p$ are distinct,  and $e_i =xy_i$
for $1\le i\le p$.
			\item [{\bf F2.}] For every edge $e_i$ with $2 \le i \le p$, there exists a
vertex $y_j$ with $1 \le j \textless i$ such that $\phiv(e_i) \in \overline{\varphi}(y_j)$.
\end{itemize}
\end{Definition}

\begin{Definition}[Kierstead path]
Let $G$ be a graph with an edge $e_1=y_0y_1$, and an edge coloring $\phiv\in \CC^{k}(G-e_1)$
for some integer $k\ge\D$.	
A Kierstead path with respect to $e_1$ and $\phiv$ is a sequence
$K=(y_0,e_1,y_1,\dots,e_p,y_p)$ with $p \ge 1$ consisting of edges $e_1, \dots,e_p$ and
vertices $y_0,y_1, \dots ,y_p$ satisfying the following two conditions$\colon$
\begin{itemize}
			\item [{\bf K1.}] The vertices $y_0, \dots, y_p$ are distinct,  and
$e_i=y_{i-1}y_i$ for $1 \le i \le p$.
			\item [{\bf K2.}] For every edge $e_i$ with $2 \le i \le p$, there exists a
vertex $y_j$  with $0 \le j \textless i$ such that $\phiv(e_i) \in
\overline{\varphi}(y_j)$.
\end{itemize}
\end{Definition}

\begin{Lemma}[Vizing's Adjacency Lemma \cite{Vizing1965}] \label{L2.1}
Let $G$ be a graph of Class $2$ with maximum degree $\D$.
If $e=xy$ is a critical edge of $G$, then $x$ has at least $\D-d_G(y)+1$ $\D$-neighbors in
$V(G) \setminus \{y\}$.	
\end{Lemma}

\begin{Lemma}[\cite{SSTF}] \label{L2.2}
Let $G$ be a graph of Class $2$ with a critical edge $e_1=xy_1$. Let $\phiv\in
\CC^{\D}(G-e_1)$ and $F=(x,e_1,y_1,\dots,e_p,y_p)$ be a multi-fan at $x$ with respect to
$e_1$ and $\phiv$. Then the following statements hold.
\begin{itemize}
			\item [$(1)$] $V(F)$ is $\phiv$-elementary;
			\item [$(2)$] If $\alpha \in \overline{\varphi}(x)$ and $\beta \in
\overline{\varphi}(y_i)$ for $ 1 \le i \le p$, then $x$ and $y_i$ are
$(\alpha,\beta)$-linked.
\end{itemize}
\end{Lemma}

\begin{Lemma}[\cite{SSTF}] \label{L2.3}
Let $G$ be a graph of Class $2$ with a critical edge $e_1=y_0y_1$. Let $\phiv\in
\CC^\Delta(G-e_1)$ and $K=(y_0,e_1,y_1,e_2,y_2,e_3,y_3)$ be a Kierstead path  with respect
to $e_1$ and $\phiv$. Then the following statements hold.
\begin{itemize}
			\item [$(1)$] If $\min\{d_G(y_1),d_G(y_2)\} \textless \Delta$, then $V(K)$ is
$\phiv$-elementary;
			\item [$(2)$] $\left| \overline{\varphi}(y_3) \cap (\overline{\varphi}(y_0)
\cup \overline{\varphi}(y_1))\right| \le 1$.	\end{itemize}
\end{Lemma}

It is easy to see that the subpaths $(y_0,e_1,y_1)$ and $(y_0,e_1,y_1,e_2,y_2)$
of the Kierstead path $K$ above are also  multi-fans with respect to $e_1$ and $\phiv$.

\begin{Lemma}[\cite{CCS2022}] \label{L2.4}
If $G$ is an $n$-vertex graph of Class $2$  with a full-deficiency pair $(a,b)$ such
that the edge $ab$ is a critical edge of $G$, then $G$ satisfies the following properties.
\begin{itemize}
			\item [$(1)$] For every $x\in(N_G(a)\cup N_G(b))\backslash\{a,b\}$,
$d_G(x)=\Delta$;
			\item [$(2)$] For every $x\in V(G)\backslash\{a,b\}$, if $d_{G}(x)\ge
n-|N_{G}(a)\cup N_{G}(b)|$, then $d_G(x)\ge\Delta-1$. If also $d_G(a)\textless\Delta$ and
$d_G(b)\textless\Delta$, then $d_G(x)=\Delta$.
\end{itemize}
\end{Lemma}

\section{The proof of Theorem 1}

The following Lemma \ref{L3.1} is an improvement of Lemma \ref{L2.3} and also an improvement of Lemma $5$ in~\cite{CCS2022}.

\begin{Lemma} \label{L3.1}
Let $G$ be a graph of Class $2$, $H\subseteq G$ with $V(H)=\{a,b,u,x,y\}$ and $E(H)=\{ab,bu,ux,uy\}$,
and let $\phiv\in\CC^\Delta(G-ab)$.
If
\[
\begin{gathered}
	K=(a,ab,b,bu,u,ux,x)\;{and}\;K^*=(a,ab,b,bu,u,uy,y)
\end{gathered}
\]
are Kierstead paths with respect to $ab$ and $\phiv$, then  $|\overline{\varphi}(x)\cap(\overline{\varphi}(a)\cup\overline{\varphi}(b))|+|\overline{\varphi}(y)\cap(\overline{\varphi}(a)\cup\overline{\varphi}(b))|\le 1$.
\end{Lemma}

\begin{proof}
Suppose to the contrary that $|\overline{\varphi}(x)\cap(\overline{\varphi}(a)\cup\overline{\varphi}(b))|+|\overline{\varphi}(y)\cap(\overline{\varphi}(a)\cup\overline{\varphi}(b))|\ge 2$. Since $K$ and $K^*$  are  Kierstead paths, it follows from  Lemma \ref{L2.3}(2) that $|\overline{\varphi}(x)\cap(\overline{\varphi}(a)\cup\overline{\varphi}(b))|=1$ and $|\overline{\varphi}(y)\cap(\overline{\varphi}(a)\cup\overline{\varphi}(b))|=1$. Then  $d_G(b)=d_G(u)=\Delta$ by Lemma \ref{L2.3}(1).
Let $\overline{\phiv}(b)=\{1\}$, $\overline{\varphi}(x)\cap(\overline{\varphi}(a)\cup\overline{\varphi}(b))=\{\eta_1\}$, $\overline{\varphi}(y)\cap(\overline{\varphi}(a)\cup\overline{\varphi}(b))=\{\eta_2\}$, $\phiv(bu)=\alpha$, $\phiv(ux)=\beta$, and $\phiv(uy)=\gamma$.
Note that $(a,ab,b)$ is a multi-fan  with respect to $ab$ and $\phiv$, and so $\{a,b\}$
is $\phiv$-elementary by Lemma \ref{L2.2}(1). We discuss the following three cases to get contradictions.

Case 1. $\eta_1=\eta_2=1$.	

In this case, $1\notin\{\alpha,\beta,\gamma\}$ and $\alpha,\beta,\gamma\in\overline{\varphi}(a)$.
Since $(a,ab,b)$ is a multi-fan  with respect to $ab$ and $\phiv$, there exists a $(1,\beta)$-chain $P_{a}(1,\beta,\phiv)=P_{b}(1,\beta,\phiv)$
with endvertices $a$ and $b$ by Lemma \ref{L2.2}(2). Note that $ux\in P_x(1,\beta,\phiv)$.
If $P_{x}(1,\beta,\phiv)=P_{y}(1,\beta,\phiv)$, then we apply a Kempe change on $P_{x}(1,\beta,\phiv)$.
If $P_{x}(1,\beta,\phiv)\neq P_{y}(1,\beta,\phiv)$, then we apply Kempe changes on $P_{x}(1,\beta,\phiv)$ and $P_{y}(1,\beta,\phiv)$, respectively.  Denote the resulting (proper) edge coloring by $\phiv_1$. Now  $\alpha,\beta,\gamma\in\overline{\phiv_1}(a)$, $\overline{\phiv_1}(b)=\{1\}$,  $\beta\in\overline{\varphi_1}(x)\cap \overline{\varphi_1}(y)$,  $\phiv_1(bu)=\alpha$, $\phiv_1(ux)=1$, and $\phiv_1(uy)=\gamma$. Note that $(a,ab,b)$ is also  a multi-fan  with respect to $ab$ and $\phiv_1$. Hence there exists a $(1,\gamma)$-chain $P_{a}(1,\gamma,\phiv_1)=P_{b}(1,\gamma,\phiv_1)$ with endvertices $a$ and $b$.

Claim $1$. $bu\in P_{x}(\alpha,\beta,\phiv_1)$ and $P_{x}(\alpha,\beta,\phiv_1)$ meets $u$ before $b$.

\begin{proof}
Let $\phiv'$ be the new (proper) edge coloring obtained from  $\phiv_1$ by coloring $ab$ with $\alpha$ and uncoloring $bu$. Now $\overline{\phiv'}(b)=\{1\}$, $\overline{\phiv'}(u)=\{\alpha\}$, $\beta\in\overline{\phiv'}(x)$, and $\phiv'(ux)=1$. Note that $ub$ is a critical edge. Therefore, $(u,ub,b,ux,x)$ is a multi-fan at $u$ with respect to $ub$ and $\phiv'$, and there exists an $(\alpha,\beta)$-chain $P_{u}(\alpha,\beta,\phiv')=P_{x}(\alpha,\beta,\phiv')$ with endvertices $u$ and $x$ by Lemma \ref{L2.2}(2). By coloring $bu$ with $\alpha$ and uncoloring $ab$, we return to the edge coloring $\phiv_1$. Thus $bu\in P_{x}(\alpha,\beta,\phiv_1)$ and $P_{x}(\alpha,\beta,\phiv_1)$ meets $u$ before $b$. The proof of Claim $1$ is finished.
\end{proof}

Claim $2$.  $ux,uy\in P_{a}(1,\gamma,\phiv_1)=P_{b}(1,\gamma,\phiv_1)$.

\begin{proof}
Suppose to the contrary that $ux,uy\notin P_{a}(1,\gamma,\phiv_1)=P_{b}(1,\gamma,\phiv_1)$.
Let $\phiv'$ be the new (proper) edge coloring obtained from $\phiv_1$ by applying a Kempe change on $P_{a}(1,\gamma,\phiv_1)=P_{b}(1,\gamma,\phiv_1)$.
Now $1,\alpha,\beta\in\overline{\phiv'}(a)$, $\overline{\phiv'}(b)=\{\gamma\}$,  $\phiv'(bu)=\alpha$, $\phiv'(ux)=1$,  and $\phiv'(uy)=\gamma$.
Note that  $P_{x}(\alpha,\beta,\phiv')=P_{x}(\alpha,\beta,\phiv_1)$ has endvertices $x$ and $z$ with $z\ne a,b$ (maybe $z=y$). By Claim $1$, let $P_{[x,u]}(\alpha,\beta,\phiv')$ be the subchain of $P_{x}(\alpha,\beta,\phiv')$ with endvertices $x$ and $u$ such that $bu \notin P_{[x,u]}(\alpha,\beta,\phiv')$.
Apply a Kempe change on $P_{[x,u]}(\alpha,\beta,\phiv')$, color $ab$ with $\alpha$, recolor $bu$: $\alpha\rightarrow\gamma$, and recolor $uy$: $\gamma\rightarrow\beta$.
Hence we get a new (proper) $\D$-edge-coloring of $G$, contradicting the fact $\chi'(G)=\D+1$.
Now the proof of Claim $2$ is finished.
\end{proof}

Now let $P_{[x,u]}(\alpha,\beta,\phiv_1)$ also be the subchain of $P_{x}(\alpha,\beta,\phiv_1)$ with endvertices $x$ and $u$ such that $bu \notin P_{[x,u]}(\alpha,\beta,\phiv_1)$.
By Claim $2$, let $P_{[u,t]}(1,\gamma,\phiv_1)$ be the subchain of $P_{a}(1,\gamma,\phiv_1)=P_{b}(1,\gamma,\phiv_1)$ with endvertices $u$ and $t$ such that $ux\in P_{[u,t]}(1,\gamma,\phiv_1)$ but $uy\notin P_{[u,t]}(1,\gamma,\phiv_1)$.
If $t=a$, then we apply a Kempe change on $P_{[x,u]}(\alpha,\beta,\phiv_1)$ and recolor $uy:\gamma\rightarrow\beta$. Denote the resulting edge coloring by $\phiv_1^*$. We still have $P_{[u,a]}(1,\gamma,\phiv_1^*)=P_{[u,a]}(1,\gamma,\phiv_1)$.
Then we apply a Kempe change on $P_{[u,a]}(1,\gamma,\phiv_1^*)$,
recolor $bu:\alpha\rightarrow 1$, and color $ab$ with $\alpha$.
Hence we get a new (proper) $\D$-edge-coloring of $G$, contradicting the fact $\chi'(G)=\D+1$.
If $t=b$, then we apply a Kempe change on $P_{[u,b]}(1,\gamma,\phiv_1)$, color $ab$ with $\gamma$, and uncolor $ux$.
Denote the resulting (proper) $\D$-edge-coloring of $G-ux$ by $\phiv_2$. Note that $ux$ is a critical edge of $G$.
Now since $\varphi_2(uy)=\gamma\in \overline{\phiv_2}(x)$, $(u,ux,x,uy,y)$ is a multi-fan at $u$ with respect to $ux$ and $\phiv_2$. However, $\beta\in\overline{\phiv_2}(x)\cap\overline{\phiv_2}(y)$ gives a contradiction to Lemma \ref{L2.2}(1).

Case 2. Just one of $\eta_1$ and $\eta_2$ is the color $1$.

Without loss of generality, we assume that $\eta_1\ne1$ and $\eta_2=1$.	In this case, $1\notin\{\alpha,\gamma,\eta_1\}$, $\alpha,\gamma,\eta_1\in\overline{\phiv}(a)$, $\beta$ may be the color $1$, and $\eta_1$ may be $\alpha$, $\gamma$.

Claim $3$. $\eta_1\neq\alpha$.

\begin{proof}
Suppose that $\eta_1=\alpha$. Note that there exist  a $(1,\alpha)$-chain $P_{a}(1,\alpha,\phiv)=P_{b}(1,\alpha,\phiv)$ and a $(1,\gamma)$-chain $P_{a}(1,\gamma,\phiv)=P_{b}(1,\gamma,\phiv)$ both with endvertices $a$ and $b$ by Lemma \ref{L2.2}(2). We have $\beta\neq1$ since otherwise, the path $(b,bu,u,ux,x)$ is a $ (1,\alpha)$-chain with endvertices $b$ and $x$, contradicting that $P_{a}(1,\alpha,\phiv)=P_{b}(1,\alpha,\phiv)$ is also a $(1,\alpha)$-chain with endvertices $a$ and $b$. Thus $\beta\in\overline{\phiv}(a)$. Note that $uy\in P_{y}(1,\gamma,\phiv)$. We apply a Kempe change on $P_{a}(1,\gamma,\phiv)$ and obtain a new  edge coloring  $\varphi'$. Now there exist an $(\alpha,\gamma)$-chain $P_{a}(\alpha,\gamma,\phiv')=P_{b}(\alpha,\gamma,\phiv')$ and a $(\beta,\gamma)$-chain $P_{a}(\beta,\gamma,\phiv')=P_{b}(\beta,\gamma,\phiv')$ both with endvertices $a$ and $b$ by Lemma \ref{L2.2}(2). Note that $bu,uy\in P_{a}(\alpha,\gamma,\phiv')$. We apply a Kempe change on $P_{x}(\alpha,\gamma,\phiv')$ and obtain a new edge coloring $\varphi''$ with $\gamma\in\overline{\phiv''}(x)$ and $ux,uy\in P_{x}(\beta,\gamma,\phiv'')$. Then we apply a Kempe change on $P_{a}(\beta,\gamma,\phiv'')$ and obtain a new  edge coloring  $\varphi'''$. Note that $1,\alpha,\gamma\in\overline{\phiv'''}(a)$, $\overline{\phiv'''}(b)=\{\beta\}$, $\phiv'''(bu)=\alpha$, $\phiv'''(ux)=\beta$, and $\phiv'''(uy)=\gamma$. Next we color $ab$ with $\alpha$, recolor $bu:\alpha\rightarrow \beta$, and uncolor $ux$. Denote the resulting edge coloring by $\phiv''''$. Note that $(u,ux,x,uy,y)$ is a multi-fan at $u$ with respect to $ux$ and $\phiv''''$ with $\overline{\phiv''''}(u)=\{\alpha\}$, $\beta,\gamma\in\overline{\phiv''''}(x)$, $1\in\overline{\phiv''''}(y)$, and $\varphi''''(uy)=\gamma$. Thus  there exist an $(\alpha,1)$-chain $P_{u}(\alpha,1,\phiv'''')=P_{y}(\alpha,1,\phiv'''')$ with endvertices $u$ and $y$,   and an $(\alpha,\beta)$-chain $P_{u}(\alpha,\beta,\phiv'''')=P_{x}(\alpha,\beta,\phiv'''')$ and an $(\alpha,\gamma)$-chain $P_{u}(\alpha,\gamma,\phiv'''')=P_{x}(\alpha,\gamma,\phiv'''')$ both with endvertices $u$ and $x$ by Lemma \ref{L2.2}(2). Note that $ab,bu\in P_{u}(\alpha,\beta,\phiv'''')$ and $ uy\in P_{u}(\alpha,\gamma,\phiv'''')$. Then we apply Kempe changes on $P_{a}(\alpha,1)$, $P_{a}(\alpha,\beta)$, $P_{a}(\alpha,\gamma)$, and $P_{a}(\alpha,1)$, sequentially. Denote the resulting edge coloring by $\phiv'''''$ with the path $(a,ab,b,bu,u)$ being an $ (\alpha,\beta)$-chain with endvertices $a$ and $u$, contradicting that $P_{u}(\alpha,\beta,\phiv''''')=P_{x}(\alpha,\beta,\phiv''''')$ is also an $(\alpha, \beta)$-chain with endvertices $u$ and $x$ by Lemma \ref{L2.2}(2). Thus we get the Claim $3$, that is, $\eta_1\neq\alpha$.
\end{proof}

Note that $\alpha,\gamma,\eta_1\in\overline{\phiv}(a)$ and $\overline{\phiv}(b)=\{1\}$. There exists a $(1,\eta_1)$-chain $P_{a}(1,\eta_1,\phiv)=P_{b}(1,\eta_1,\phiv)$ with endvertices $a$ and $b$ by Lemma \ref{L2.2}(2). If $P_{x}(1,\eta_1,\phiv)\ne P_{y}(1,\eta_1,\phiv)$, that is, the other endvertex of $P_{x}(1,\eta_1,\phiv)$ is some vertex $z\in V(G)\backslash V(H)$, then applying a Kempe change on $P_{x}(1,\eta_1,\phiv)$  results in that $1$ is one missing color at $x$. Thus we are in the previous Case $1$. Next we need to consider the other case $P_{x}(1,\eta_1,\phiv)=P_{y}(1,\eta_1,\phiv)$. Note that $\beta$ may be the color $1$, and $\eta_1$ may be $\gamma$ by Claim $3$. We discuss the following three subcases to get contradictions.

If $ux,uy\in P_{x}(1,\eta_1,\phiv)$ (implying $\beta=1$ and $\eta_1=\gamma$), then we color $ab$ with $\alpha$, recolor $bu:\alpha\rightarrow 1$, and uncolor $ux$.  Denote the resulting (proper) edge coloring by $\phiv_1$. Clearly, $1,\eta_1\in\overline{\phiv_1}(x)$,
$1\in\overline{\phiv_1}(y)$, and $\phiv_1(uy)=\eta_1$. Hence $(u,ux,x,uy,y)$ is a multi-fan at $u$ with respect to $ux$ and $\phiv_1$, but $1\in\overline{\varphi_1}(x)\cap\overline{\varphi_1}(y)$, which is a contradiction by Lemma \ref{L2.2}(1).

If $ux\in P_{x}(1,\eta_1,\phiv)$ and $uy\notin P_{x}(1,\eta_1,\phiv)$ (implying $\beta=1$ and $\eta_1\neq\gamma$), then we apply a Kempe change on $P_{a}(1,\gamma,\phiv)$=$P_{b}(1,\gamma,\phiv)$, color $ab$ with $\alpha$, recolor $bu:\alpha\rightarrow\gamma$, and uncolor $uy$. Denote the resulting (proper) edge coloring by $\phiv_1$.
Clearly, $1,\eta_1\in\overline{\varphi_1}(a)$, $\eta_1\in\overline{\varphi_1}(x)$, $1,\gamma\in\overline{\varphi_1}(y)$, $\overline{\varphi_1}(u)=\{\alpha\}$, $\phiv_1(ab)=\alpha$, $\phiv_1(bu)=\gamma$, and $\phiv_1(ux)=1$.
Now $(u,uy,y,ux,x)$ is a  multi-fan at $u$ with respect to $uy$ and $\phiv_1$.
Thus  there exist an $(\alpha,1)$-chain $P_{u}(\alpha,1,\phiv_1)=P_{y}(\alpha,1,\phiv_1)$ and an $(\alpha,\gamma)$-chain $P_{u}(\alpha,\gamma,\phiv_1)=P_{y}(\alpha,\gamma,\phiv_1)$ both with endvertices $u$ and $y$, and  there also exists an $(\alpha,\eta_1)$-chain $P_{u}(\alpha,\eta_1,\phiv_1)=P_{x}(\alpha,\eta_1,\phiv_1)$ with endvertices $u$ and $x$ by Lemma \ref{L2.2}(2).
Note that $ab, bu\in P_{u}(\alpha,\gamma,\phiv_1)$ and $ux\in P_{u}(\alpha,1,\phiv_1)$. Then we apply  Kempe changes on $P_{a}(\alpha,1)$, $P_{a}(\alpha,\gamma)$, $P_{a}(\alpha,\eta_1)$, and $P_{a}(\alpha,1)$, sequentially. Denote the resulting (proper) edge coloring by $\phiv_2$.
Note that $1,\gamma\in\overline{\phiv_2}(a)\cap\overline{\phiv_2}(y) $, $\overline{\phiv_2}(u)=\{\alpha\}$, $\phiv_2(ab)=\alpha$, and $\phiv_2(bu)=\gamma$.
Hence the path $(u,ub,b,ba,a)$ is an $ (\alpha,\gamma)$-chain with endvertices $u$ and $a$, contradicting that $P_{u}(\alpha,\gamma,\phiv_2)=P_{y}(\alpha,\gamma,\phiv_2)$ is also an $(\alpha, \gamma)$-chain with endvertices $u$ and $y$ by Lemma \ref{L2.2}(2).

If $ux\notin P_{x}(1,\eta_1,\phiv)$ (implying $\beta\neq1$),
then we claim that we can obtain a new (proper) edge coloring $\phiv_1$ from $\phiv$ such that $\alpha,\beta,\gamma\in\overline{\varphi_1}(a)$ (also $\eta_1\in\overline{\varphi_1}(a)$ if $uy\notin P_{x}(1,\eta_1,\phiv)$),  $\overline{\varphi_1}(b)=\{1\}$, $1\in\overline{\varphi_1}(x)$, $\gamma\in\overline{\varphi_1}(y)$, $\phiv_1(bu)=\alpha$, $\phiv_1(ux)=\beta$, and $\phiv_1(uy)=1$.
We consider the following two subcases to get the  claim above.
Note that there exist a $(1,\alpha)$-chain $P_a(1,\alpha,\phiv)=P_b(1,\alpha,\phiv)$, a $(1,\gamma)$-chain  $P_a(1,\gamma,\phiv)=P_b(1,\gamma,\phiv)$, and a $(1,\eta_1)$-chain  $P_a(1,\eta_1,\phiv)=P_b(1,\eta_1,\phiv)$ all with endvertices $a$ and $b$.
If $uy\in P_{x}(1,\eta_1,\phiv)$ (implying $\eta_1=\gamma$), then we apply a Kempe change on $P_{x}(1,\eta_1,\phiv)=P_{y}(1,\eta_1,\phiv)$. Denote the resulting (proper) edge coloring by $\phiv_1$, as desired.
If $uy\notin P_{x}(1,\eta_1,\phiv)$ (implying $\eta_1\neq\gamma$), then we apply Kempe changes on $P_{a}(1,\eta_1)$, $P_{y}(1,\gamma)$, $P_{x}(\alpha,\eta_1)$, $P_{a}(\gamma,\eta_1)$, $P_{a}(1,\gamma)$, and $P_{x}(1,\alpha)$, sequentially. We also get the resulting (proper) edge coloring $\phiv_1$, as desired.

Note that there exists a $(1,\beta)$-chain $P_{a}(1,\beta,\phiv_1)=P_{b}(1,\beta,\phiv_1)$ with endvertices $a$ and $b$, and $ux,uy\in P_x(1,\beta,\phiv_1)$. Let $\phiv_2$ be the (proper) edge coloring obtained from $\phiv_1$ by applying a Kempe change on $P_{a}(1,\beta,\phiv_1)$, coloring $ab$ with $\alpha$, recoloring $bu:\alpha\rightarrow\beta$, and uncoloring $ux$.
Now $(u,ux,x,uy,y)$ is a multi-fan at $u$ with respect to $ux$ and $\phiv_2$ with $\overline{\varphi_2}(u)=\{\alpha\}$, $1,\beta\in\overline{\varphi_2}(x)$, $\gamma\in\overline{\varphi_2}(y)$, and $\phiv_2(uy)=1$. By Lemma \ref{L2.2}(2), there exist an $(\alpha,1)$-chain $P_{u}(\alpha,1,\phiv_2)=P_{x}(\alpha,1,\phiv_2)$ and an $(\alpha,\beta)$-chain $P_{u}(\alpha,\beta,\phiv_2)=P_{x}(\alpha,\beta,\phiv_2)$ both with endvertices $u$ and $x$, and there also exists an $(\alpha,\gamma)$-chain $P_{u}(\alpha,\gamma,\phiv_2)=P_{y}(\alpha,\gamma,\phiv_2)$ with endvertices $u$ and $y$. Note that $ab,bu\in P_{u}(\alpha,\beta,\phiv_2)$ and  $uy\in P_{u}(\alpha,1,\phiv_2)$. Then we apply Kempe changes on  $P_{a}(\alpha,1)$, $P_{a}(\alpha,\beta)$, $P_{a}(\alpha,\gamma)$, and $P_{a}(\alpha,1)$, sequentially.
Denote the resulting (proper) edge coloring by $\phiv_3$.  Note that $1,\beta\in\overline{\phiv_3}(a)\cap \overline{\phiv_3}(x)$, $\overline{\phiv_3}(u)=\{\alpha\}$, $\phiv_3(ab)=\alpha$, and $\phiv_3(bu)=\beta$.
Hence the path $(u,ub,b,ba,a)$ is an $ (\alpha,\beta)$-chain with endvertices $u$ and $a$, contradicting that $P_{u}(\alpha,\beta,\phiv_3)=P_{x}(\alpha,\beta,\phiv_3)$ is also an $(\alpha, \beta)$-chain with endvertices $u$ and $x$ by Lemma \ref{L2.2}(2).

Case 3. $\eta_1\ne1$ and $\eta_2\ne1$.	

In this case, $\eta_1,\eta_2\in\overline{\phiv}(a)$. By Lemma \ref{L2.2}(2), there exist a $(1,\eta_1)$-chain $P_{a}(1,\eta_1,\phiv)=P_{b}(1,\eta_1,\phiv)$ and a $(1,\eta_2)$-chain $P_{a}(1,\eta_2,\phiv)=P_{b}(1,\eta_2,\phiv)$ both with endvertices $a$ and $b$. Apply a Kempe change on $P_{y}(1,\eta_2,\phiv)$.
If $P_{y}(1,\eta_2,\phiv)=P_{x}(1,\eta_1,\phiv)$ (in this case $\eta_2=\eta_1$), then the color $1$ is one missing color at both $x$ and $y$, and so we are in the previous Case $1$. If $P_{y}(1,\eta_2,\phiv)\neq P_{x}(1,\eta_1,\phiv)$, then the color $1$ is one missing color at $y$, and so we are in the previous Case $2$.
This completes the proof of Lemma \ref{L3.1}.
\end{proof}

Now we give the proof of Theorem \ref{T1} as follows.

\begin{proof}

It is easy to see that a regular Class $1$ graph has even order, and that every graph obtained
from a regular graph of even order by a vertex-splitting is overfull.
Thus $n\ge 5$ and $G$ is overfull, and so $G$ is of Class $2$.
Then, in order to prove that $G$ is $\D$-critical, we only need to show that every edge of
$G$ is critical.
Suppose to the contrary that there exists $xy\in E(G)$ such that $xy$ is not a critical edge
of $G$.
Let $G^*=G-xy$. Thus $\chi'(G^*)=\Delta+1$.
Let $ab$ be the edge of $G$ whose contraction gets an $(n-1)$-vertex $\Delta$-regular Class
$1$ graph. Thus $a$ and $b$ have distinct neighbors, and each vertex in $V(G)\backslash\{a,b\}$ has degree $\Delta$ in $G$. Furthermore, $(a,b)$ is a full-deficiency pair in $G$ such
that $ab$ is a critical edge, which implies $ab\ne xy$.
Since $G^*-ab\subset G-ab$, we have $\Delta\le\chi'(G^*-ab)\le\chi'(G-ab) =\Delta$, which
implies that  $\chi'(G^*-ab)=\Delta$, and so $ab$ is also a critical edge of $G^*$. Let $\phiv\in \CC^\Delta(G^*-ab)$.

First suppose that $\{a,b\}\cap\{x,y\}\ne\emptyset$. Since $ab\ne xy$, we say $a=x$ and $b\neq y$.
Then  $d_{G^*}(a)+d_{G^*}(b)=\Delta+1$.
However, $d_{G^*}(a)+d_{G^*}(b)\ge\Delta+2$ by Lemma \ref{L2.1}, a contradiction.

Therefore, $\{a,b\}\cap\{x,y\}=\emptyset$.
Then $d_{G^*}(x)=d_{G^*}(y)=\Delta-1$ and $|\overline{\varphi}(x)|=|\overline{\varphi}(y)|=1$.
Since $d_{G^*}(a)+d_{G^*}(b)=d_{G}(a)+d_{G}(b)=\D+2$,  we have $(a,b)$ is a full-deficiency pair in $G^*$, and $x,y\notin N_{G^*}(a)\cup N_{G^*}(b)$ by Lemma \ref{L2.4}(1). It follows that $d_{G^*}(x)+|N_{G^*}(a)\cup N_{G^*}(b)|=\D-1+\D+2=2\D+1\ge n$ since $\D\ge \frac{2n-4}{3}$. Thus we have one of $a$ and $b$ has degree $\D$ by Lemma \ref{L2.4}(2), say $d_{G^*}(b)=\D$, which implies $d_{G^*}(a)=2$.
Note that $ab$ is a critical edge of $G^*$.
If $c$ is the other neighbor of $a$ in $G^*$, then  $d_{G^*}(c)=\Delta$ by Lemma
\ref{L2.1}.
Thus $(a,c)$ is also a full-deficiency pair of $G^*$.
If $\overline{\phiv}(b)=\{1\}$,
then $\phiv(ac)=1$ since $(a,ab,b)$ is a multi-fan and $\{a,b\}$ is $\phiv$-elementary.
Let $\phiv'$ be obtained from $\phiv$ by coloring $ab$ with the color $1$ and uncoloring
$ac$.
Note that $\overline{\phiv'}(a)=\overline{\phiv}(a)$, $\overline{\phiv'}(b)=\emptyset$, and
$\overline{\phiv'}(c)=\{1\}$.
Now $\phiv'$ is a (proper) $\D$-edge-coloring of $G^*-ac$.
Since $\chi'(G^*)=\D+1$, $ac$ is also a critical edge of $G^*$.
Similarly, we have $x,y\notin N_{G^*}(c)$ by Lemma \ref{L2.4}(1).

Since $b,c,x,y\notin N_{G^*}(b)$ and $a,b,c,x,y\notin N_{G^*}(x)\cup N_{G^*}(y)$, we have
$|N_{G^*}(b)\cap N_{G^*}(x)\cap N_{G^*}(y)|=|N_{G^*}(b)|+|N_{G^*}(x)\cap N_{G^*}(y)|-|N_{G^*}(b)\cup (N_{G^*}(x)\cap N_{G^*}(y))|=|N_{G^*}(b)|+|N_{G^*}(x)|+|N_{G^*}(y)|-|N_{G^*}(x)\cup N_{G^*}(y)|-|N_{G^*}(b)\cup (N_{G^*}(x)\cap N_{G^*}(y))|\ge\D+2(\D-1)-(n-5)-(n-4)=3\D-2n+7\textgreater1$. Thus there exists one vertex $u\in N_{G^*}(b)\cap N_{G^*}(x)\cap N_{G^*}(y)$.
Now we can find a subgraph $H$ with $V(H)=\{a,b,u,x,y\}$ and $E(H)=\{ab,bu,ux,uy\}$.
As $\{a,b\}$ is $\phiv$-elementary,
$|\overline{\varphi}(a)\cup\overline{\varphi}(b)|=|\overline{\varphi}(a)|+|\overline{\varphi}(b)|=\Delta$
and so $\overline{\varphi}(a)\cup\overline{\varphi}(b)=[\Delta]$.
Thus in the subgraph $H$, $K=(a,ab,b,bu,u,ux,x)$ and $K^*=(a,ab,b,bu,u,uy,y)$ are Kierstead
paths, but $|\overline{\varphi}(x)\cap(\overline{\varphi}(a)\cup\overline{\varphi}(b))|+ |\overline{\varphi}(y)\cap(\overline{\varphi}(a)\cup\overline{\varphi}(b))|=|\overline{\varphi}(x)|+|\overline{\varphi}(y)|=2$, contradicting Lemma \ref{L3.1}.
Now the proof of Theorem \ref{T1} is finished.
\end{proof}

\noindent {\bf Acknowledgements}

The authors are very grateful to the two anonymous reviewers for their valuable comments.
This work was supported by Hebei Natural Science Foundation A2023208006, Hebei Fund for Introducing Overseas Returnees C20230357, and Foundation of China Scholarship Council 202508130088.


\begin{thebibliography}{1}	
	\bibitem{BV}
	S. Bonvicini, A. Vietri, {\it A M\"obius-type gluing technique for obtaining edge-critical
		graphs}, Ars Math. Contemp.  19 (2)(2020) 209--229.
	
	\bibitem{CCS2022}
	Y. Cao, G. Chen, S. Shan, {\it An improvement to the Hilton-Zhao vertex-splitting
		conjecture}, Discrete Math. 345 (2022) 112902.
	
	\bibitem{FW}
	S. Fiorini, R.J. Wilson, {\it Edge-Colourings of Graphs}, Research Notes in Maths., Pitman,
	London, 1977.	
	
	\bibitem{HZ}
	A.J.W. Hilton, C. Zhao, {\it Vertex-splitting and chromatic index critical graphs}, Discrete
	Appl. Math.  76 (1997) 205--211.
	
	\bibitem{Holyer}
	I. Holyer, {\it The NP-completeness of edge-coloring}, SIAM J. Comput. 10 (4)(1981) 718--720.
	
	\bibitem{Gupta}
	R.G. Gupta, {\it Studies in the Theory of Graphs}, PhD dissertation, Tata Institute of Fundamental
	Research, Bombay, 1967.
	
	\bibitem{Song}
	Z. Song, {\it A further extension of Yap's construction for $\D$-critical graphs}, Discrete
	Math.  243 (1-3)(2002) 283--290.
	
	\bibitem{SSTF}
	M. Stiebitz,  D. Scheide, B. Toft,  L.M. Favrholdt, {\it Graph edge coloring}, Wiley
	Series in Discrete Mathematics and Optimization. John Wiley \& Sons, Inc., Hoboken, NJ, 2012.
	Vizing's theorem and Goldberg's conjecture, With a preface by Stiebitz and Toft.
	
	\bibitem{Vietri2015}
	A. Vietri, {\it An analogy between edge colourings and differentiable manifolds, with a new
	perspective on 3-critical graphs}, Graphs Comb.  31 (6)(2015) 2425--2435.
	
	\bibitem{Vizing1965}
	V.G. Vizing, {\it Critical graphs with given chromatic class}, Diskretn. Anal. 5 (1965)
	9--17.
\end{thebibliography}
\end{document}